\documentclass[11pt]{article}
\usepackage{amsthm, amsmath, amssymb, amsfonts, url, booktabs, tikz, setspace, fancyhdr, bm}
\usepackage{geometry}
\usepackage{enumerate}
\usepackage[shortlabels]{enumitem}
\usepackage[babel]{microtype}
\usepackage[english]{babel}
\usepackage{comment}
\usepackage{bbm}
\usepackage{bm}
\usepackage{csquotes}
\usepackage{mathabx}
\usetikzlibrary{calc, angles, quotes, intersections}
\usepackage{graphicx}
\usepackage{float}
\usepackage{xcolor}

\counterwithin{figure}{section}

\newtheorem{theorem}{Theorem}[section]

\newtheorem{lemma}[theorem]{Lemma}

\newtheorem{claim}[theorem]{Claim}

\theoremstyle{definition}

\newlist{Case}{enumerate}{2}
\setlist[Case, 1]{%
   label={\bfseries Case \arabic*.},
   labelindent=1em,labelwidth=1.3cm,labelsep*=1em,leftmargin=!
}
\setlist[Case, 2]{%
   label={\bfseries Subcase \arabic{Casei}.\arabic*.},
   labelindent=-1em,labelwidth=1.3cm,labelsep*=1em,leftmargin=!
}

\newenvironment{poc}{\begin{proof}[Proof of claim]}{\end{proof}}

\newcommand{\V}[1]{\bm{#1}}

\newcommand{\me}{{\mathsf e}}
\usepackage{hyperref}

\usepackage[capitalise]{cleveref}

\crefname{prop}{proposition}{propositions}
\Crefname{prop}{Proposition}{Propositions}
\crefname{claim}{claim}{claims}
\Crefname{claim}{Claim}{Claims}
\crefname{cor}{corollary}{corollaries}
\Crefname{cor}{Corollary}{Corollaries}
\crefname{defn}{definition}{definitions}
\Crefname{defn}{Definition}{Definitions}

\newcommand{\T}{\mathbb{T}}
\newcommand{\Z}{\mathbb{Z}}
\newcommand{\R}{\mathbb{R}}
\newcommand{\Q}{\mathbb{Q}}
\newcommand{\Prob}{\mathbb{P}}

\title{An improved lower bound for the van der Waerden number \(w(3,k)\)}
\author{Haitao Cao\thanks{School of Mathematical Sciences, Zhejiang University. Email:3230104589@zju.edu.cn.}}
\date{}

\begin{document}

\maketitle

\begin{abstract}
For an integer \(k\ge 3\), let \(w(3,k)\) be the least \(n\) such that every red-blue coloring of \([n]\) contains either a nontrivial three-term arithmetic progression in blue or a nontrivial \(k\)-term arithmetic progression in red. A recent breakthrough of Green proved that \(w(3,k)\ge k^{\Omega\left(\left(\frac{\log k}{\log\log k}\right)^{1/3}\right)}\) when \(k\) is large, and Hunter later improved the bound to \(w(3,k)\ge k^{\Omega\left(\frac{\log k}{\log\log k}\right)}\). On the other hand, Green remarked that it is reasonable to believe that \(w(3,k)\le k^{O(\log k)}\). We prove that \(w(3,k)\ge k^{\Omega(\log k)}\), which perhaps gives some evidence that \(k^{\Theta(\log k)}\) is the correct order of magnitude.
\end{abstract}

\section{Introduction}

For an integer \(k\ge 3\), a \(k\)-AP is a nontrivial arithmetic progression of length \(k\). Let \(w(3,k)\) be the least positive integer \(n\) such that for any red-blue coloring of \([n]\), there is either a blue 3-AP or a red k-AP. Van der Waerden~\cite{VanDerWaerden1927} proved that \(w(3,k)\) is finite. From then on, the order of \(w(3,k)\) for large \(k\) has been widely studied.

For many years, the evidence pointed towards polynomial growth. Brown, Landman and Robertson~\cite{BrownLandmanRobertson2008} proved a near-quadratic lower bound, and Li and Shu~\cite{LiShu2010} sharpened it to \(w(3,k)\ge k^{2-o(1)}\). As Green~\cite{Green2022} mentioned, Graham had conjectured the quadratic upper bound \(w(3,k)=O(k^{2})\), a prediction supported by the numerical data then available. A recent breakthrough of Green~\cite{Green2022} overturned this picture. He proved the first superpolynomial lower bound
\[
w(3,k)\ge k^{\Omega\left(\frac{\log k}{\log\log k}\right)^{1/3}}.
\]
Hunter~\cite{Hunter2022} later simplified the construction and improved the estimate to \(w(3,k)\ge k^{\Omega\left(\frac{\log k}{\log\log k}\right)}\). For a broader overview of the problem and Green's breakthrough, we refer interested readers to Klarreich's article~\cite{Klarreich2021}.

In the other direction, upper bounds for \(w(3,k)\) can be obtained through the classical Roth-type problem concerning the maximal size of subsets of \([n]\) without nontrivial three-term arithmetic progressions. Let
\[
r_3(n)=\max\{|A|:A\subseteq [n],\ A\text{ contains no nontrivial }3\text{-AP}\}.
\]
Indeed, an upper bound for \(r_3(n)\) can provide an upper bound for \(w(3,k)\). More precisely, if \(r_3(n)\le n\exp\left(-\Omega((\log n)^s)\right)\), then \(w(3,k)\le \exp\left(O((\log k)^{1/s})\right)\). To see this, suppose that \([n]\) is colored red and blue, and let \(B\) denote the blue set. If \(B\) contains no nontrivial three-term arithmetic progression, then by definition \(|B|\le r_3(n)\). Hence \(\frac{|B|}{n}\le \exp\left(-\Omega((\log n)^s)\right)\), and therefore the red set \(R=[n]\setminus B\) has density at least \(\frac{|R|}{n}\ge 1-\exp\left(-\Omega((\log n)^s)\right)\). Choose the minimal \(n\) such that \(\log k=o((\log n)^s)\). Then the blue density is \(o(1)\), so the red density is \(1-o(1)\). By Szemerédi's theorem~\cite{Szemeredi1975}, the red set contains a \(k\)-term arithmetic progression, and hence \(w(3,k)\le n\). Since \(\log k=o((\log n)^s)\), we have \(\log n=\Theta((\log k)^{1/s})\), and therefore \(w(3,k)\le \exp\left(O((\log k)^{1/s})\right)\). Notice that the current best known upper bound for \(r_3(n)\) is \(r_3(n)\le n\exp\left(-\Omega\left(\frac{(\log n)^{1/6}}{\log\log n}\right)\right),\) which was obtained by Raghavan~\cite{Raghavan2026}. Through the above density argument, this yields the upper bound \(w(3,k)\le \exp\left(O\left((\log k)^6(\log\log k)^6\right)\right).\) The historical development of quantitative improvements on the upper bounds for \(r_3(n)\) can be found in ~\cite{BloomSisask2023,Bourgain1999,2023StrongBounds3Progressions,Raghavan2026,Roth1953,Roth1954,Roth1970,Roth1972,Sanders2011}. 

Green~\cite{Green2022} remarked that it is reasonable to expect \(w(3,k)\le k^{O(\log k)}\), as the Behrend construction~\cite{Behrend1946} for \(3\)-AP-free sets is believed to be asymptotically sharp for the Roth problem and suggests this order of magnitude, also see a recent improvement~\cite{ElsholtzHunterProskeSauermann2024}.
Building on Hunter's work~\cite{Hunter2022}, we prove the corresponding lower bound using a more refined Fourier control argument.

\begin{theorem}\label{thm:main}
There is an absolute constant \(c>0\) such that, for all sufficiently large \(k\),
\[
w(3,k)\ge k^{c\log k}.
\]
\end{theorem}

\section{Proof of Theorem~\ref{thm:main}}

Our proof is based on the torus constructions of
Green~\cite{Green2022} and Hunter~\cite{Hunter2022}.
We map \(n\in[N]\) to \(n\V{\theta}\in\T^r\) and color \(n\)
blue when its image lies in one of several thin shells.
Green~\cite{Green2022} used random ellipsoidal shells. Hunter~\cite{Hunter2022} simplified the
argument by choosing the shell radii independently, giving
each long progression many independent chances to contain
a blue point.

Our main ingredient is to work inside boxes whose side lengths
do not depend on the dimension. These boxes admit a product
test function with nonnegative Fourier coefficients.
By selecting frequencies according to the sizes of their
coefficients, we keep only exponentially many frequencies,
with an error smaller than the zero-frequency term.
A measure estimate for the images of boxes in quotient tori
then allows us to choose centers so that small shifts make
the relevant resonant phases zero.
Together with Fej\'er weights, this shows that every
sufficiently long progression meets many of the boxes.

The centers and the rotation are chosen so that any blue
\(3\)-AP must lie in a single shell. The parallelogram
identity would then force its common difference on the torus
to be too close to zero, contrary to the choice of rotation.
To rule out red \(k\)-APs, we first fix the points where each
long progression meets the boxes, and then choose the shell
radii independently. A union bound shows that, with positive
probability, every such progression contains a blue point.
This finally removes the \(\log\log k\) loss
in Hunter's bound~\cite{Hunter2022}.

\subsection{Background and useful tools}\label{subsec:tools}

We map the sequence into the torus \(\T^{r}\), where \(\T=\R/\Z\) is viewed as \([0,1)\) with addition modulo one. For a positive integer \(r\), let \(m_{\T^{r}}\) denote normalized Haar measure on \(\T^{r}\), namely its uniform probability measure. The dimension \(r\) is the only scale parameter in the construction; eventually \(r=\Theta(\log k)\), and all \(o(1)\)-terms refer to \(r\to\infty\). Whenever a positive real expression below is required to be an integer, we round it up. We suppress the ceiling signs, since they have no effect on the estimates.

For \(t\in\T\), let \(\|t\|_{\T}\) be its distance to the nearest integer. For \(\V{x}=(x_{1},\ldots,x_{r})\in\T^{r}\), set \(\|\V{x}\|_{\T^{r},\infty}=\max_{1\le j\le r}\|x_{j}\|_{\T}\). If \(0<s<\frac{1}{2}\), we use \(Q_{s}\) for both \((-s,s)^{r}\subseteq\R^{r}\) and its injective image in \(\T^{r}\). Under this identification, \(\|\cdot\|_{2}\) denotes the Euclidean norm on \(Q_{s}\). 

For a fixed \(\V{\theta}\in\T^{r}\), the map \(n\to n\V{\theta}\) sends an AP \(\{n_{0}+jd:0\le j<\ell\}\) to the orbit segment \(\{n_{0}\V{\theta}+j(d\V{\theta}):0\le j<\ell\}\). We later choose a set \(B\subseteq\T^{r}\) and color \(n\) blue exactly when \(n\V{\theta}\in B\). Additive restrictions on B will rule out blue 3-APs, by further ensuring that sufficiently long orbit segments must intersect B, which can forbid red k-APs.

Write \(\me(t)=\exp(2\pi it)\). The characters of \(\T^{r}\) are \(\V{x}\to\me(\V{\xi}\cdot\V{x})\), indexed by \(\V{\xi}\in\Z^{r}\), and our Fourier convention is \(\widehat f(\V{\xi})=\int_{\T^{r}}f(\V{x})\me(-\V{\xi}\cdot\V{x})\,d\V{x}\).

We use the following lattice facts. Call \(L\subseteq\Z^{r}\) saturated if \(q\V{v}\in L\), with \(q\) a positive integer and \(\V{v}\in\Z^{r}\), always forces \(\V{v}\in L\). If \(\V{\lambda}_{1},\ldots,\V{\lambda}_{m}\) is a \(\Z\)-basis of a saturated lattice, then
\(
\V{x}\to(\V{\lambda}_{1}\cdot\V{x},\ldots,\V{\lambda}_{m}\cdot\V{x})
\)
maps \(\T^{r}\) onto \(\T^{m}\); this follows from the Smith normal form. We also use that a continuous surjective homomorphism between compact tori sends normalized Haar measure to normalized Haar measure.

We need a function supported on a short interval with nonnegative Fourier coefficients. Write \(u_{+}=\max\{u,0\}\), and define \(\psi(x)=(1-400\|x\|_{\T})_{+}\). A direct calculation gives \(\widehat\psi(0)=\frac{1}{400}\) and \(\widehat\psi(n)=\frac{400\sin^{2}\left(\frac{\pi n}{400}\right)}{\pi^{2}n^{2}}\) for \(n\ne0\). Thus \(\widehat\psi(n)\ge0\), \(\widehat\psi(n)=O((1+n^{2})^{-1})\), and \(\sum_{n\in\Z}\widehat\psi(n)=\psi(0)=1\). For \(\V{x}=(x_{1},\ldots,x_{r})\in\T^{r}\), set \(\Psi_{r}(\V{x})=\prod_{j=1}^{r}\psi(x_{j})\). Then
\[
\widehat\Psi_{r}(\V{\xi})=\prod_{j=1}^{r}\widehat\psi(\xi_{j})\ge 0,\sum_{\V{\xi}\in\Z^{r}}\widehat\Psi_{r}(\V{\xi})=1,\widehat\Psi_{r}(0)=400^{-r}.
\]
Its Fourier series converges absolutely and uniformly. Moreover, \(\Psi_{r}(\V{x})>0\) exactly when \(\V{x}\in Q_{\frac{1}{400}}\), so \(\Psi_r(\V{x})\) can be viewed as a characteristic function that helps us tell whether \(\V{x}\) lies in the small box \(Q_{\frac{1}{400}}\).

Lemma~\ref{lem:spectrum} shows that almost all of the Fourier mass is carried by exponentially many frequencies, with an error smaller than a fixed fraction of the zero coefficient.

\begin{lemma}\label{lem:spectrum}
For every sufficiently large \(r\), there is a finite set \(S_{r}\subseteq\Z^{r}\) such that
\begin{enumerate}
    \item[\textup{(1)}] \(\V{0}\in S_{r}.\)
    \item[\textup{(2)}] \(|S_{r}|\le \exp(100r).\)
    \item[\textup{(3)}] \(\sum_{\V{\xi}\notin S_{r}}\widehat\Psi_{r}(\V{\xi})<\frac{400^{-r}}{100}.\)
\end{enumerate}
\end{lemma}

\begin{proof}[Proof of Lemma~\ref{lem:spectrum}]
Since \(\sum_{n\in\Z}\widehat\psi(n)^{3/4}<\infty\), this sum is a finite constant. Therefore \(\log\left(\sum_{n\in\Z}\widehat\psi(n)^{3/4}\right)\) is finite, and hence \(\frac{100}{4}-\log\left(\sum_{n\in\Z}\widehat\psi(n)^{3/4}\right)>\log 400+2\) holds for the fixed constant \(100\). Define \(S_{r}=\{\V{\xi}\in\Z^{r}:\widehat\Psi_{r}(\V{\xi})\ge \exp(-100r)\}.\) If \(\V{\xi}\notin S_{r}\), then \(\widehat\Psi_{r}(\V{\xi})\le\exp\left(-\frac{100r}{4}\right)\widehat\Psi_{r}(\V{\xi})^{3/4}.\) Summing this inequality and using the product structure gives
\[
\sum_{\V{\xi}\notin S_{r}}\widehat\Psi_{r}(\V{\xi})
\le \exp\left(-\frac{100r}{4}\right)\left(\sum_{n\in\Z}\widehat\psi(n)^{3/4}\right)^{r}
\le \exp\left(-\left(\log 400+2\right)r\right)
<\frac{400^{-r}}{100}
\]
for all sufficiently large \(r\). Since all Fourier coefficients are nonnegative and sum to one, \(|S_{r}|\le \exp(100r)\). Moreover, \(100>\log 400\), so \(\V{0}\in S_{r}\). This proves the lemma.
\end{proof}

We need a lower bound for the image of a small box under a rational quotient map. Let \(V\subseteq\Q^{r}\) have dimension \(m\), and put \(L_{V}=V\cap\Z^{r}\), a saturated lattice of rank \(m\). For a \(\Z\)-basis \(\V{\lambda}_{1},\ldots,\V{\lambda}_{m}\) of \(L_{V}\), define the surjective homomorphism \(\phi_{V}:\T^{r}\longrightarrow\T^{m}\) by \(\phi_{V}(\V{x})=(\V{\lambda}_{1}\cdot\V{x},\ldots,\V{\lambda}_{m}\cdot\V{x})\).

\begin{lemma}\label{lem:box-image}
If \(0<t<\frac{1}{2}\), then \(m_{\T^{m}}(\phi_{V}(Q_{t}))\ge (2t)^{m}\).
\end{lemma}

\begin{proof}[Proof of Lemma~\ref{lem:box-image}]
Let \(A\) be the \(m\times r\) integer matrix whose rows are \(\V{\lambda}_{1},\ldots,\V{\lambda}_{m}\). Choose a set \(J\) of \(m\) coordinates such that the corresponding square minor \(A_{J}\) is nonsingular, and consider the copy of \((-t,t)^{m}\) in \(Q_{t}\) supported on \(J\). Since \(t<\frac{1}{2}\), this box embeds in \(\T^{r}\), and the restriction of \(\phi_{V}\) to it is \(\V{u}\to A_{J}\V{u}\pmod{\Z^{m}}\).

The Euclidean image \(A_{J}(-t,t)^{m}\) has volume \(|\det A_{J}|(2t)^{m}\). The torus endomorphism induced by \(A_{J}\) has kernel of order \(|\det A_{J}|\). Hence every point of \(\T^{m}\) has at most \(|\det A_{J}|\) preimages in \((-t,t)^{m}\). Applying the area formula before and after reduction modulo \(\Z^{m}\) gives \(|\det A_{J}|(2t)^{m}\le |\det A_{J}|m_{\T^{m}}(\phi_{V}(Q_{t}))\), which proves the result.
\end{proof}

We need centers that are sufficiently additively separated, with one further property: for every low-dimensional space generated by \(S_{r}\), the corresponding phases can be corrected at many centers.

\begin{lemma}\label{lem:centers}
For every sufficiently large \(r\), there are \(\exp\left(\frac{r}{10}\right)\) points \(\V{z}_{i}\in\T^{r}\), indexed by \(i\in[\exp\left(\frac{r}{10}\right)]\), with the following properties.
\begin{enumerate}
    \item[\textup{(1)}] If \(i,j,\ell\in[\exp\left(\frac{r}{10}\right)]\) are not all equal, then \(\V{z}_{i}-2\V{z}_{j}+\V{z}_{\ell}\notin Q_{\frac{1}{25}}\).
    \item[\textup{(2)}] Let \(V\subseteq\Q^{r}\) be generated by elements of \(S_{r}\), and suppose that \(\dim V<\frac{r}{1000}\). For every \(\V{x}\in\T^{r}\), there are \(\exp\left(\frac{r}{100}\right)\) pairwise distinct indices \(i\) and corresponding vectors \(\V{u}_{i}\in Q_{\frac{1}{400}}\) such that \(\V{\xi}\cdot(\V{x}-\V{z}_{i}-\V{u}_{i})=0\) in \(\T\) for every \(\V{\xi}\in L_{V}\).
\end{enumerate}
\end{lemma}

\begin{proof}[Proof of Lemma~\ref{lem:centers}]
Choose the \(\exp\left(\frac{r}{10}\right)\) points independently and uniformly at random according to Haar measure on \(\T^{r}\).

\begin{claim}\label{claim:center-separation}
With probability \(1-o(1)\), the first property of Lemma~\ref{lem:centers} holds.
\end{claim}

\begin{poc}
For any \(i,j,\ell\) that are not all equal, the random variable \(\V{z}_{i}-2\V{z}_{j}+\V{z}_{\ell}\) is Haar distributed on \(\T^{r}\). This is still true when \(i=\ell\ne j\), since multiplication by \(2\) keeps Haar measure unchanged. The union bound gives
\[
\Prob \bigl(\text{the separation property fails}\bigr)\le \exp\left(\frac{r}{3}\right)\left(\frac{2}{25}\right)^{r}=o(1).
\]
\end{poc}

\begin{claim}\label{claim:center-covering}
With probability \(1-o(1)\), the second property of Lemma~\ref{lem:centers} holds for all spaces \(V\) generated by elements of \(S_{r}\) with \(\dim V<\frac{r}{1000}\).
\end{claim}

\begin{poc}
Fix such a space \(V\) of dimension \(m<\frac{r}{1000}\), and put \(W_{V}=\phi_{V}(Q_{\frac{1}{1600}})\). We can choose a maximal finite set \(T_{V}\subseteq\T^{m}\) such that the sets \(\V{t}+W_V\) and \(\V{t}'+W_V\) are disjoint for any distinct \(\V{t},\V{t}'\in T_V\). This process stops because Lemma~\ref{lem:box-image} gives \(|T_{V}|m_{\T^{m}}(W_{V})\le 1\) and \(m_{\T^{m}}(W_{V})\ge\left(\frac{1}{800}\right)^{m}\), so \(|T_{V}|\le800^{m}\). For every \(\V{q}\in\T^{m}\), maximality forces \(\V{q}+W_{V}\) to meet some \(\V{t}+W_{V}\). Hence \(\V{q}\in \V{t}+(W_{V}-W_{V})\subseteq \V{t}+\phi_{V}(Q_{\frac{1}{800}})\).

Split the indices into \(\exp\left(\frac{r}{100}\right)\) disjoint blocks, each containing at least \(\exp\left(\frac{r}{20}\right)\) indices when \(r\) is large. Fix one block and one \(\V{t}_{0}\in T_{V}\). By Lemma~\ref{lem:box-image}, the set \(\V{t}_{0}-\phi_{V}(Q_{\frac{1}{800}})\) has measure at least \(400^{-m}\). The probability that no center in the block maps into this set is at most \(\exp\left(-\exp\left(\frac{r}{25}\right)\right).\)

Every such \(V\) has a basis contained in \(S_{r}\), so there are at most \(\sum_{0\le j<\frac{r}{1000}}|S_{r}|^{j}=\exp(O(r^{2}))\) such spaces. Taking a union bound over all these spaces, all \(\exp\left(\frac{r}{100}\right)\) blocks and all \(\V{t}_{0}\in T_{V}\), the total failure probability is still \(o(1)\). We may therefore assume that every block contains a suitable center for every \(V\) and every \(\V{t}_{0}\in T_{V}\).

Now fix \(\V{x}\in\T^{r}\). Choose \(\V{t}_{0}\in T_{V}\) and \(\V{w}_{1}\in\phi_{V}(Q_{\frac{1}{800}})\) such that \(\phi_{V}(\V{x})=\V{t}_{0}+\V{w}_{1}\). In each block choose a center \(\V{z}_{i}\) for which \(\phi_{V}(\V{z}_{i})=\V{t}_{0}-\V{w}_{2,i}\) for some \(\V{w}_{2,i}\in\phi_{V}(Q_{\frac{1}{800}})\). Then \(\phi_{V}(\V{x}-\V{z}_{i})=\V{w}_{1}+\V{w}_{2,i}\in\phi_{V}(Q_{\frac{1}{400}})\), so there is \(\V{u}_{i}\in Q_{\frac{1}{400}}\) with \(\phi_{V}(\V{x}-\V{z}_{i}-\V{u}_{i})=0\). The chosen indices lie in different blocks and are therefore distinct. Since the coordinate functions of \(\phi_{V}\) form a basis of \(L_{V}\), the required phase identity follows.
\end{poc}

Claims~\ref{claim:center-separation} and~\ref{claim:center-covering} hold together with positive probability, so a deterministic choice of the centers has both properties.
\end{proof}

For the orbit average, inspired by Hunter~\cite{Hunter2022}, we use the following Fej\'er weight. For \(n\in\Z\), define \(\omega(n)=\frac{\bigl(\exp(200r)-|n|\bigr)_{+}}{\bigl(\exp(200r)\bigr)^{2}}\). Then \(\omega\ge 0\) and \(\sum_{n\in\Z}\omega(n)=1\).

\begin{lemma}\label{lem:fejer}
For every \(t\in\T\),
\[
\widehat\omega(t):=\sum_{n\in\Z}\omega(n)\me(nt)=\left|\frac{1}{\exp(200r)}\sum_{j=0}^{\exp(200r)-1}\me(jt)\right|^{2}\ge 0,
\]
and \(\widehat\omega(t)\le1\). If \(t\ne0\) in \(\T\), then \(\widehat\omega(t)\le\frac{1}{\bigl(2\exp(200r)\|t\|_{\T}\bigr)^{2}}\).
\end{lemma}

\begin{proof}[Proof of Lemma~\ref{lem:fejer}]
For the first identity, write \(\omega\) as the convolution of the uniform probability measure on \(\{0,\ldots,\exp(200r)-1\}\) with its reflection. This also gives the trivial upper bound \(1\). If \(t\ne0\) in \(\T\), the geometric sum gives
\[
\left|\frac{1}{\exp(200r)}\sum_{j=0}^{\exp(200r)-1}\me(jt)\right|\le\frac{1}{2\exp(200r)\|t\|_{\T}},
\]
where we used \(\bigl|\sin(\pi t)\bigr|\ge2\|t\|_{\T}\).
\end{proof}

For \(\V{\alpha}\in\T^{r}\), define its resonant set by
\[
R(\V{\alpha})=\left\{\V{\xi}\in S_{r}:\|\V{\xi}\cdot\V{\alpha}\|_{\T}<\frac{10\cdot20^{r}}{\exp(200r)}\right\}.
\]

These are the important frequencies. On them, the phase \(\V{\xi}\cdot\V{\alpha}\) is nearly zero, so the terms do not oscillate much along the orbit. Outside \(R(\V{\alpha})\), the phase grows quickly with \(n\), so these terms oscillate fast and have average zero. 

We need one rotation that works for every possible common difference. For each \(d\), the resonant frequencies of \(d\V{\theta}\) should span a low-dimensional space, and \(d\V{\theta}\) itself should be long enough.

\begin{lemma}\label{lem:rotation}
For every sufficiently large \(r\), there is \(\V{\theta}\in\T^{r}\) such that for every \(d\in[\exp\left(\frac{r^{2}}{1000}\right)]\),
\begin{enumerate}
    \item[\textup{(1)}] \(\dim\operatorname{span}_{\Q}R(d\V{\theta})<\frac{r}{1000}.\)
    \item[\textup{(2)}]\(\|d\V{\theta}\|_{\T^{r},\infty}>\exp\left(-\frac{r}{400}\right).\)
\end{enumerate}

\end{lemma}

\begin{proof}[Proof of Lemma~\ref{lem:rotation}]
Choose \(\V{\theta}\) according to Haar measure on \(\T^{r}\). If \(\dim\operatorname{span}_{\Q}R(d\V{\theta})\ge\frac{r}{1000}\), then \(R(d\V{\theta})\) contains \(\frac{r}{1000}\) linearly independent frequencies \(\V{\xi}_{1},\ldots,\V{\xi}_{\frac{r}{1000}}\in S_{r}\). For each fixed independent tuple and each fixed \(d\), the homomorphism
\[
\V{\theta}\longrightarrow\left(\V{\xi}_{1}\cdot d\V{\theta},\ldots,\V{\xi}_{\frac{r}{1000}}\cdot d\V{\theta}\right)
\]
maps \(\T^{r}\) onto \(\T^{\frac{r}{1000}}\). Indeed, its dual sends an integer vector \(\left(q_{1},\ldots,q_{\frac{r}{1000}}\right)\) to \(d\sum_{j=1}^{\frac{r}{1000}}q_{j}\V{\xi}_{j}\), and this map is injective. The resulting point of \(\T^{\frac{r}{1000}}\) is therefore Haar distributed. The union bound gives
\[
\begin{split}
&\Prob\left(\exists d\in\left[\exp\left(\frac{r^{2}}{1000}\right)\right]:\dim\operatorname{span}_{\Q}R(d\V{\theta})\ge\frac{r}{1000}\right)\\
&\le\exp\left(\frac{r^{2}}{1000}\right)|S_{r}|^{\frac{r}{1000}}\left(\frac{20^{r+1}}{\exp(200r)}\right)^{\frac{r}{1000}}\\
&\le\exp\left(-\frac{r^{2}}{10000}\right)=o(1).
\end{split}
\]
Here the last estimate follows from \(|S_{r}|\le\exp(100r)\).

For each fixed \(d\), the point \(d\V{\theta}\) is also Haar distributed. A union bound over \(d\in[\exp\left(\frac{r^{2}}{1000}\right)]\) shows that the probability of \(\|d\V{\theta}\|_{\T^{r},\infty}\le\exp\left(-\frac{r}{400}\right)\) for some such \(d\) is at most
\[
\exp\left(\frac{r^{2}}{1000}\right)\left(2\exp\left(-\frac{r}{400}\right)\right)^{r}\le\exp\left(-\frac{r^{2}}{1000}\right).
\]
The two bad events have total probability smaller than one for all sufficiently large \(r\), which proves the lemma.
\end{proof}

The last lemma says that if the phases match exactly on the resonant frequencies, every long enough orbit reaches \(Q_{\frac{1}{100}}\). The Fej\'er weight keeps the aligned resonant terms nonnegative, so they cannot cancel the zero-frequency term.

\begin{lemma}\label{lem:orbit}
Let \(\V{\alpha}\in\T^{r}\), put \(V=\operatorname{span}_{\Q}R(\V{\alpha})\), and suppose that \(\dim V<\frac{r}{1000}\). Let \(\V{x},\V{z}\in\T^{r}\). If there is \(\V{u}\in Q_{\frac{1}{400}}\) such that \(\V{\xi}\cdot(\V{x}-\V{z}-\V{u})=0\) in \(\T\) for every \(\V{\xi}\in L_{V}\), then there is an integer \(n\) with \(|n|<\exp(200r)\) such that \(\V{x}+n\V{\alpha}-\V{z}\in Q_{\frac{1}{100}}\).
\end{lemma}

\begin{proof}[Proof of Lemma~\ref{lem:orbit}]
Consider the nonnegative real number \(S=\sum_{n\in\Z}\omega(n)\Psi_{r}(\V{x}+n\V{\alpha}-\V{z}-\V{u})\). Absolute convergence permits termwise Fourier expansion, giving
\[
S=\sum_{\V{\xi}\in\Z^{r}}\widehat\Psi_{r}(\V{\xi})\me\bigl(\V{\xi}\cdot(\V{x}-\V{z}-\V{u})\bigr)\widehat\omega(\V{\xi}\cdot\V{\alpha}).
\]
The zero frequency contributes \(400^{-r}\). If \(\V{\xi}\in R(\V{\alpha})\), then \(\V{\xi}\in V\cap\Z^{r}=L_{V}\), so the phase is one. By Lemma~\ref{lem:fejer}, the remaining factor is nonnegative. If \(\V{\xi}\in S_{r}\setminus R(\V{\alpha})\), then Lemma~\ref{lem:fejer} and the definition of \(R(\V{\alpha})\) give \(|\widehat\omega(\V{\xi}\cdot\V{\alpha})|\le400^{-r-1}\).
For \(\V{\xi}\notin S_{r}\), we only use \(|\widehat\omega(\V{\xi}\cdot\V{\alpha})|\le1\). Since \(S\) is real, taking real parts and using Lemma~\ref{lem:spectrum} gives
\[
S\ge400^{-r}-400^{-r-1}\sum_{\V{\xi}\in S_{r}}\widehat\Psi_{r}(\V{\xi})-\sum_{\V{\xi}\notin S_{r}}\widehat\Psi_{r}(\V{\xi})>400^{-r}\left(1-\frac{1}{400}-\frac{1}{100}\right)>0.
\]
Both \(\omega\) and \(\Psi_{r}\) are nonnegative. Hence \(\Psi_{r}(\V{x}+n\V{\alpha}-\V{z}-\V{u})>0\) for some integer \(n\) with \(|n|<\exp(200r)\). The positivity set of \(\Psi_{r}\) is \(Q_{\frac{1}{400}}\), so \(\V{x}+n\V{\alpha}-\V{z}-\V{u}\in Q_{\frac{1}{400}}\). Since \(\V{u}\in Q_{\frac{1}{400}}\), we obtain \(\V{x}+n\V{\alpha}-\V{z}\in Q_{\frac{1}{200}}\subseteq Q_{\frac{1}{100}}\).
\end{proof}

\subsection{Completion of the proof}\label{subsec:completion}

\begin{proof}[Proof of Theorem~\ref{thm:main}]
Let \(k\) be sufficiently large, and choose an integer \(r\) such that \(\frac{\log k}{500}\le r\le\frac{\log k}{400}\). In particular, \(2\exp(200r)-1\le k\).

Fix centers \(\V{z}_{i}\), indexed by \(i\in[\exp\left(\frac{r}{10}\right)]\), as in Lemma~\ref{lem:centers}, and choose \(\V{\theta}\) as in Lemma~\ref{lem:rotation}. Independently for each \(i\), choose \(s_{i}\) uniformly at random from \(\{0,\ldots,r\cdot\exp\left(\frac{r}{200}\right)-1\}\). Define
\[
A_{i}=\left\{\V{x}\in Q_{\frac{1}{100}}:s_{i}\cdot\exp\left(-\frac{r}{200}\right)\le\|\V{x}\|_{2}^{2}<(s_{i}+1)\exp\left(-\frac{r}{200}\right)\right\},
\]
and put \(B=\bigcup_{i=1}^{\exp\left(\frac{r}{10}\right)}(\V{z}_{i}+A_{i})\). Color \(n\in[\exp\left(\frac{r^{2}}{1000}\right)]\) blue if \(n\V{\theta}\in B\), and red otherwise.

\begin{claim}\label{claim:no-blue}
For every choice of the \(s_{i}\), there is no blue \(3\)-AP.
\end{claim}

\begin{poc}
Suppose that \(n-d,n,n+d\in[\exp\left(\frac{r^{2}}{1000}\right)]\) are blue for some \(d\ge1\). Choose \(i,j,\ell\) and \(\V{r}_{1}\in A_{i},\V{r}_{2}\in A_{j},\V{r}_{3}\in A_{\ell}\) so that \((n-d)\V{\theta}=\V{z}_{i}+\V{r}_{1}\), \(n\V{\theta}=\V{z}_{j}+\V{r}_{2}\), and \((n+d)\V{\theta}=\V{z}_{\ell}+\V{r}_{3}\) in \(\T^{r}\). Taking second differences gives \(\V{z}_{i}-2\V{z}_{j}+\V{z}_{\ell}=-(\V{r}_{1}-2\V{r}_{2}+\V{r}_{3})\in Q_{\frac{1}{25}}\).
Lemma~\ref{lem:centers} gives \(i=j=\ell\). Hence \(\V{r}_{1}-2\V{r}_{2}+\V{r}_{3}\) is an integer vector all of whose coordinates have absolute value less than \(\frac{1}{25}\), and so it is zero. Write \(\V{r}_{1}=\V{r}_{2}-\V{v}\) and \(\V{r}_{3}=\V{r}_{2}+\V{v}\).
Since \(\V{r}_{1},\V{r}_{2},\V{r}_{3}\) lie in the same shell, the parallelogram identity gives
\[
\|\V{v}\|_{2}^{2}=\frac{\|\V{r}_{1}\|_{2}^{2}+\|\V{r}_{3}\|_{2}^{2}}{2}-\|\V{r}_{2}\|_{2}^{2}<\exp\left(-\frac{r}{200}\right).
\]
But \(d\V{\theta}=\V{r}_{2}-\V{r}_{1}=\V{v}\) in \(\T^{r}\). Since \(\V{v}\in Q_{\frac{1}{50}}\subseteq\left(-\frac{1}{2},\frac{1}{2}\right)^{r}\), it is the representative of \(d\V{\theta}\) in \(\left(-\frac{1}{2},\frac{1}{2}\right)^{r}\), and \(\|d\V{\theta}\|_{\T^{r},\infty}\le\|\V{v}\|_{2}<\exp\left(-\frac{r}{400}\right)\), contrary to Lemma~\ref{lem:rotation}.
\end{poc}

\begin{claim}\label{claim:no-red}
With positive probability, the coloring has no red AP of length \(2\exp(200r)-1\).
\end{claim}

\begin{poc}
Fix an AP \(P\) of this length, with common difference \(d\ge1\), say
\[
P=\{n_{0}+jd:0\le j\le2\exp(200r)-2\}\subseteq\left[\exp\left(\frac{r^{2}}{1000}\right)\right].
\]
Put \(\V{x}=(n_{0}+(\exp(200r)-1)d)\V{\theta}\), \(\V{\alpha}=d\V{\theta}\), and \(V=\operatorname{span}_{\Q}R(\V{\alpha})\). By Lemma~\ref{lem:rotation}, \(\dim V<\frac{r}{1000}\), while \(V\) is generated by elements of \(S_{r}\). By Lemma~\ref{lem:centers}, we have \(\exp\left(\frac{r}{100}\right)\) distinct indices \(i\) and vectors \(\V{u}_{i}\in Q_{\frac{1}{400}}\) such that \(\V{\xi}\cdot(\V{x}-\V{z}_{i}-\V{u}_{i})=0\) for every \(\V{\xi}\in L_{V}\). For each such \(i\), there exists an integer \(n_{i}\) with \(|n_{i}|<\exp(200r)\) such that\(\V{q}_{i}:=\V{x}+n_{i}\V{\alpha}-\V{z}_{i}\in Q_{\frac{1}{100}},\) by Lemma~\ref{lem:orbit}. Since \(0\le\exp(200r)-1+n_{i}\le2\exp(200r)-2\), we have \(n_{0}+(\exp(200r)-1+n_{i})d\in P\), and its image on the torus is \(\V{x}+n_{i}\V{\alpha}=\V{z}_{i}+\V{q}_{i}\).

For each \(i\), fix one such \(n_{i}\), and hence \(\V{q}_{i}\). None of these choices depends on the variables \(s_{i}\). Since
\[
0\le\|\V{q}_{i}\|_{2}^{2}<\frac{r}{10000}<r\cdot\exp\left(\frac{r}{200}\right)\exp\left(-\frac{r}{200}\right),
\]
there is a unique \(t_{i}\in\{0,\ldots,r\cdot\exp\left(\frac{r}{200}\right)-1\}\) such that \(t_{i}\cdot\exp\left(-\frac{r}{200}\right)\le\|\V{q}_{i}\|_{2}^{2}<(t_{i}+1)\exp\left(-\frac{r}{200}\right)\).
Now \(s_{i}=t_{i}\) with probability \(\frac{1}{r\cdot\exp\left(\frac{r}{200}\right)}\), and on this event \(P\) contains a blue point. These events are independent, since the indices \(i\) are distinct. Thus
\[
\Prob(P\text{ is entirely red})\le\left(1-\frac{1}{r\cdot\exp\left(\frac{r}{200}\right)}\right)^{\exp\left(\frac{r}{100}\right)}\le\exp\left(-\frac{\exp\left(\frac{r}{100}\right)}{r\cdot\exp\left(\frac{r}{200}\right)}\right).
\]
The quotient in the exponent is \(\exp(\Omega(r))\), and there are at most \(\exp(O(r^{2}))\) APs in the ambient interval. A union bound therefore shows that, with probability \(1-o(1)\), none is entirely red.
\end{poc}

By Claims~\ref{claim:no-blue} and~\ref{claim:no-red}, there is a choice of the \(s_{i}\) for which the coloring has no blue \(3\)-AP and no red AP of length \(2\exp(200r)-1\). Since this length is at most \(k\), the same coloring has no red \(k\)-AP. The coloring is defined on an interval of length \(\exp\left(\frac{r^{2}}{1000}\right)\), so \(w(3,k)\ge\exp\left(\frac{r^{2}}{2000}\right)\) for all sufficiently large \(r\). Since \(r=\Theta(\log k)\), this gives \(w(3,k)\ge k^{c\log k}\) for some absolute constant \(c>0\), as required.
\end{proof}

\section*{Acknowledgement}
The author would like to thank Tao Feng, Zixiang Xu, Haihua Deng, Xing Zhang, Weiye Zhang, Junchi Zhang, and Yuan Chang for helpful discussions.

\bibliographystyle{abbrv}
\bibliography{Vander}

\end{document}